\documentclass[10pt,leqno]{amsart}
\usepackage{amssymb,amsthm}
\usepackage{amsmath}
\usepackage{setspace}
\usepackage{dcpic,pictexwd}
\usepackage[all]{xy}
\usepackage[pdftex]{graphicx}
\usepackage{mathrsfs}
\usepackage[pdftex]{hyperref}
\usepackage{bbm}
\usepackage{tikz}
\usetikzlibrary{cd}

\theoremstyle{definition}
 \newtheorem{definition}{Definition}[section]

\theoremstyle{plain}

\theoremstyle{plain}
 \newtheorem{theorem}[definition]{Theorem}
 
\theoremstyle{definition}

\theoremstyle{plain}
 \newtheorem{lemma}[definition]{Lemma}

\theoremstyle{plain}

\theoremstyle{remark}
 \newtheorem{remark}[definition]{Remark}

\theoremstyle{definition}

\theoremstyle{plain}

\newcommand{\Hom}{\mathrm{Hom}}

\newcommand{\Ob}{\mathrm{Ob}}

\newcommand{\m}{\mathfrak{m}}
\renewcommand{\k}{\Bbbk}

\newcommand{\Cb}{\mathscr{C}}

\hypersetup{
    bookmarks=true,         
    unicode=false,          
    pdftoolbar=true,        
    pdfmenubar=true,        
    pdffitwindow=false,     
    pdfstartview={FitH},    
    pdfnewwindow=true,      
    colorlinks=true,       
    linkcolor=red,          
    citecolor=blue,        
    filecolor=magenta,      
    urlcolor=blue           
}

\title[liftings of Persistence Modules]{Liftings of point-wise finite dimensional persistence modules over local commutative Artinian rings} 
\thanks{}

\author[V\'elez-Marulanda]{Jos\'e A. V\'elez-Marulanda}
\address{Department of Applied Mathematics \& Physics, Valdosta State University, Valdosta, GA,  United States of America}
\email{javelezmarulanda@valdosta.edu}
\address{Facultad de Matem\'aticas e Ingenier\'{\i}as, Fundaci\'on Universitaria Konrad Lorenz, Bogot\'a D.C.,  Colombia}
\email{josea.velezm@konradlorenz.edu.co}


\begin{document}
\maketitle

\begin{abstract}
Let $\k$ be a field and let $V: \Cb \to \k\textup{-Mod}$ be a pointwise finite dimensional persistence module, where $\Cb$ is a small category. Assume that for all local Artinian $\k$-algebras  $R$ with residue field isomorphic to $\k$, there is a generalized persistence module $M: \Cb \to R\textup{-Mod}$, such that for all $x\in \Ob(\Cb)$, $M(x)$ is free over $R$ with finite rank and $\k\otimes_R M(x)\cong V(x)$. If $V$ is a direct sum of indecomposable persistence modules $V_I: \Cb\to  \k\textup{-Mod}$ with endomorphism ring isomorphic to $\k$, then $M$ is a direct sum of indecomposables  $M_I:\Cb\to R\textup{-Mod}$ with endomorphism ring isomorphic to $R$. 
\end{abstract}

\renewcommand{\labelenumi}{\textup{(\roman{enumi})}}
\renewcommand{\labelenumii}{\textup{(\roman{enumi}.\alph{enumii})}}
\renewcommand{\labelenumiii}{\textup{(\roman{enumi}.\alph{enumii}.\arabic{enumiii})}}

\numberwithin{equation}{section}

\section{Introduction}\label{into}

In this article, we assume that $\k$ is a field of arbitrary characteristic and denote by $\mathcal{A}_\k$ the category of local commutative Artinian $\k$-algebras with residue field isomorphic to $\k$. In particular, the morphisms in $\mathcal{A}_\k$ induce the identity on $\k$.  For all objects $R$ in $\mathcal{A}_\k$, we denote by $R$-Mod the category of $R$-modules. In particular, $\k$-Mod is just the category of $\k$-vector spaces. Let $\Cb$ be a small category and let $V: \Cb\to \k\textup{-Mod}$ be a {\it persistence module indexed by $\Cb$}, i.e., $V$ is a functor from $\Cb$ to $\k\textup{-Mod}$. Recall that $V$ is said to be {\it pointwise finite dimensional} if for all objects $x\in \Ob(\Cb)$, $V(x)$ has finite dimension over $\k$.  It follows from a result due to M. B. Botman and W. Crawley-Boevey (see \cite[Thm. 1.1]{botman1}) that if $V$ is pointwise finite dimensional then $V$ is a direct sum of indecomposable modules with local endomorphism ring. This result improved that due to Crawley-Boevey (see \cite[Thm. 1.1]{crawley-boevey}) which states that every pointwise persistence module over a totally ordered indexing set $(\mathscr{R}, <)$ (with $\mathscr{R}$ having a countable subset which is dense in the order topology on $\mathscr{R}$) is a direct sum of indecomposable persistence modules each having an endomorphism ring isomorphic to $\k$. They also prove in \cite[Thm. 1.3]{botman1} that every pointwise finite dimensional middle exact module over a product of two totally ordered sets (as defined in \cite[\S 2]{cochoy} decomposes into a direct sum of {\it block modules} (as defined in \cite[\S 2]{cochoy}), which provides an alternative proof to \cite[Thm. 2.1]{cochoy}. It is noted in the Proof Outline section in \cite{cochoy} that these block modules have endomorphism ring isomorphic to $\k$. These observations provide non-trivial examples of pointwise finite dimensional persistence modules that decompose as a direct sum of indecomposable persistence modules with endomorphism ring isomorphic to $\k$.  

Let $R$ be a fixed object in $\mathcal{A}_\k$. An {\it $R$-generalized persistence module indexed by $\Cb$} is just a functor $M:\Cb\to R\textup{-Mod}$. Let $V$ be a pointwise finite dimensional persistence module. We say that $(M, \phi)$ is a {\it lift} of $V$ over $R$ is a $R$-generalized persistence module $M$ such that for  all objects $x\in \Ob(\Cb)$, $M(x)$ is free of finite rank over $R$, together with a natural equivalence $\phi: (\k\otimes_R-)\circ M\to V$, where $\k\otimes_R-: R\textup{-Mod}\to \k\textup{-Mod}$ is the natural projection functor. We say that two lifts $(M,\phi)$ and $(M',\phi')$ of $V$ over $R$ are {\it isomorphic} if there is a natural equivalence $f: M\to M'$ such that the induced natural equivalence $\mathrm{id}_\k\otimes f: (\k\otimes_R-)\circ M\to (\k\otimes_R-)\circ M'$ satisfies that $\phi'\circ (\mathrm{id}_\k\otimes f) = \phi$. Our goal is to give a short proof of the following result.

\begin{theorem}\label{thm1}
Let $V$ a pointwise finite dimensional persistence module such that $V$ is the direct sum of indecomposable persistence modules with endomorphism ring isomorphic to $\k$. Then for all objects $R$ in $\mathcal{A}_\k$ and all lifts $(M,\phi)$ of $V$ over $R$, the $R$-generalized persistence module $M$ is a direct sum of indecomposables $R$-generalized persistence modules with endomorphism ring isomorphic to $R$. 
\end{theorem}

\begin{remark}
Observe that if $V$ is a pointwise finite dimensional persistence module, then $V$ has always a {\it trivial lift} which is defined by the composition of  $V$ with the functor $(R\otimes_\k -): \k\textup{-Mod}\to R\textup{-Mod}$.  However, if for example we let $\Cb$ be the cycle quiver with relations $(Q,\rho)$, where   
\begin{align*}
Q: &\xymatrix{\underset{1}{\bullet}\ar@(ul,dl)_{\gamma}}\text{ and } \rho=\{\gamma^2\},
\end{align*}
then the simple persistence module $S_1$ corresponding to the unique vertex of $Q$ fits in a short exact sequence of persistence modules 
\begin{equation*}
0\to S_1\xrightarrow{\iota} P_1\xrightarrow{\pi}S_1\to 0,
\end{equation*} 
where $P_1$ is the projective cover of $S_1$. It follows that $P_1$ defines a non-trivial lift of $S_1$ over the ring of dual numbers $\k[\epsilon]$, with $\epsilon^2=0$, by letting $\epsilon$ act on $P_1$ as $\iota \circ \pi$.  
\end{remark}

It is important to mention that generalized persistence modules over symmetric monoidal categories, which include categories of vector spaces, have been studied by A. Patel in \cite{patel}.  

\section{Proof of Theorem \ref{thm1}}
Let $R$ be a fixed object in $\mathcal{A}_\k$. We denote by $\textbf{Pers}^{\Cb}_R$ the abelian category if $R$-generalized persistence modules indexed by $\Cb$. Note that in particular $\textbf{Pers}^{\Cb}_\k$ is just the the category of persistence modules $V: \Cb\to \k\textup{-Mod}$. If $\Cb(x,y)$ is a morphism between objects in $\Cb$, then for all $R$-generalized persistence modules $M$ in $\textbf{Pers}^{\Cb}_R$, we let $M(x,y) = M(\Cb(x,y))$. Note that every object $V$ in  $\textbf{Pers}^{\Cb}_\k$ can be viewed as an object in $\textbf{Pers}^{\Cb}_R$ by considering $\k$ as an $R$-module. For all objects $M$ and $N$ in $\textbf{Pers}^{\Cb}_R$, we denote by $\textbf{Hom}^\Cb_R(M,N)$ the abelian group of morphisms (i.e. natural transformations) $f: M\to N$. Note that in particular that $\textbf{Hom}^\Cb_R(M,N)$ is an $R$-module. We also let $\textbf{End}^\Cb_R(M)= \textbf{Hom}^\Cb_R(M,M)$ and denote by $\k\otimes_RM$ the composition of functors $(\k\otimes_R -)\circ M$. For all lifts $(M,\phi)$ of $V$ over $R$, we denote by $\pi_M: M\to V$ the natural projection, i.e., for all $x\in \Cb$, $(\pi_M)_x: M(x)\to V(x)$  is the natural projection between $R$-modules. 

\begin{lemma}\label{lemma2.1}

Let $V$ be a pointwise finite dimensional persistence module and let $(M,\phi)$ be a lift of $V$ over $R$. Then there is an isomorphism of $R$-modules $\mathbf{Hom}^\Cb_R(M,V)\cong \mathbf{Hom}^\Cb_\k(V,V)$. 
\end{lemma}
\begin{proof}
Let $f\in \mathbf{Hom}^\Cb_\k(V,V)$ and let $x\in\Ob(\Cb)$ be fixed but arbitrary. Since $M(x)$ is free over $R$, it follows that there exists $g_x\in \Hom_R(M(x),V(x))$ such that $g_x = f_x\circ (\pi_M)_x$. It is straightforward to prove that for all morphisms $\Cb(x,y)$ in $\Cb$, $g_y\circ M(x,y)=V(x,y)\circ g_x$, and thus we have a morphism $g\in \mathbf{Hom}^\Cb_R(M,V)$. In this way, we obtain a morphism between $R$-modules $\Phi: \mathbf{Hom}^\Cb_\k(V,V)\to \mathbf{Hom}^\Cb_R(M,V)$, defined by $\Phi(f) = g$. Next let $g\in  \mathbf{Hom}^\Cb_R(M,V)$. For all $x\in \Ob(\Cb)$, let $f_x: V(x)\to V(x)$ be defined as $f_x(v) = g_x(v)$ for all $v\in V(x)$. Thus we obtain a morphism $f\in  \mathbf{Hom}^\Cb_\k(V,V)$, and in this way we obtain a morphism between $R$-modules $\Psi:  \mathbf{Hom}^\Cb_R(M,V)\to  \mathbf{Hom}^\Cb_\k(V,V)$ defined as $\Psi(g)=f$. It is straightforward to prove that $\Phi$ and $\Psi$ are inverses of each other and thus  $\mathbf{Hom}^\Cb_R(M,V)\cong \mathbf{Hom}^\Cb_\k(V,V)$.

\end{proof}

\begin{lemma}\label{lemma1.2}
Assume that $V: \Cb\to \k\textup{-Mod}$ is a pointwise finite dimensional persistence module such that $\mathbf{End}^\Cb_\k(V) =\k$. Then for all objects $R$ in $\mathcal{A}_\k$ and all lifts $(M, \phi)$ of $V$ over $R$, $\mathbf{End}^\Cb_R(M)= R$. In particular, $M$ is an indecomposable $R$-generalized persistence module.    
\end{lemma}

In order to prove Lemma \ref{lemma1.2}, we need to recall the following definition from \cite[Def. 1.2]{sch}.

\begin{definition}\label{defi3.7}
Let $\theta:R\to R_0$ be a morphism of Artinian objects in $\mathcal{A}_\k$. We say that $\theta$ is a {\it small extension} if the kernel of $\theta$ is a non-zero principal ideal $tR$ that is annihilated by the unique maximal ideal $\m_R$ of $R$.
\end{definition}

\begin{remark}
\begin{enumerate}
\item Let $\theta: R\to R_0$ be a small extension in $\mathcal{A}_\k$ with $tR = \ker \theta$. Then it follows that $tR\cong \k$ as $R$-modules. 
\item Let $R$ be a fixed object in $\mathcal{A}_\k$ with length $\ell(R)>0$. Then there is a small extension $\theta: R\to R_0$. Indeed, let $\m_R$ the unique maximal ideal of $R$. Since $R$ is Artinian, there exists  $p\geq 1$ such that $\m_R^p=0$ and $\m_R^{p-1}\not=0$. Let $t\in \m_R^{p-1}$ with $t\not=0$. By letting $R_0 = R/tR$, we obtain that $R_0$ is in $\mathcal{A}_\k$ and the natural projection $\theta:R\to R_0$ is a small extension. Therefore, we can use small extensions to prove results on the the objects $R$ in $\mathcal{A}_\k$ by using induction on the length  of $R$. 
\end{enumerate}
\end{remark}

\begin{proof}[Proof of Lemma \ref{lemma1.2}]
In the following, we adapt the arguments in the proof of \cite[Lemma 2.3]{bleher8} to our situation. 
Let $\theta: R\to R_0$ be a small extension in $\mathcal{A}_\k$ with $\ker \theta = tR$ and assume that for all lifts $(M_0, \phi_0)$ of $V$ over $R_0$, we have $\mathbf{End}^\Cb_{R_0}(M_0)=R_0$. Let $f\in \mathbf{End}^\Cb_R(M)$, and consider $f_0 = \mathrm{id}_{R_0}\otimes_{R,\theta}f\in \mathbf{End}^\Cb_{R_0}(R_0\otimes_{R,\theta}M)$. Since $R_0\otimes_{R,\theta}M$ defines a lift of $V$ over $R_0$, it follows by induction that there exists $r_0\in R_0$ such that $f_0 = \mu_{r_0}$, where $\mu_{r_0}$ denotes multiplication by $r_0$. Let $r\in R$ such that $\theta(r)=r_0$ and consider $g = f-\mu_r\in \mathbf{End}^\Cb_R(M)$. Then $\mathrm{id}_{R_0}\otimes_{R,\theta}g=0$. This implies that $g \in \mathbf{Hom}^\Cb_R(M,tR\otimes_RM)$. Since $tR\cong \k$, it follows that $tR\otimes_RM\cong \k\otimes_R M\cong V$. By using Lemma \ref{lemma2.1}, we have that $\mathbf{Hom}^\Cb_R(M,tR\otimes_RM)\cong \mathbf{End}^\Cb_\k(V)$. Since $\mathbf{End}^\Cb_\k(V)=\k$, there exists $\lambda \in R$ such that $g= \mu_{t\lambda}$. This implies that $f = \mu_{r+t\lambda}$ and thus every morphism in $\mathbf{End}^\Cb_R(M)$ is multiplication by an scalar in $R$ and thus  $\mathbf{End}^\Cb_R(M) = R$.
\end{proof}

\begin{proof}[Proof of Theorem \ref{thm1}]
Assume that there is an isomorphism $\overline{\varphi}: V \to \bigoplus_{I\in \mathscr{I}}V_I$, with $\mathbf{End}^\Cb_\k(V_I)=\k$ for all $I\in \mathscr{I}$. Let $(M, \phi)$ be a lift of $V$ over $R$, and for all  $I\in \mathscr{I}$ let $(M_I, \phi_I)$ be a lift of $V_I$ over $R$.  For each $x\in \Ob(\Cb)$, there is a morphism $\overline{\varphi}_x: V(x)\to \bigoplus_{I\in \mathscr{I}}V_I(x)$, where $V_I(x) \not=0$ for at most finitely many $I\in \mathscr{I}$. Since $M(x)$ is free over $R$, there exists a morphism of $R$-modules $\varphi_x: M(x)\to \bigoplus_{I\in \mathscr{I}}M_I(x)$ such that $\overline{\varphi}_x\circ (\pi_M)_x = \left(\bigoplus_{I\in \mathscr{I}}\pi_I(x)\right)\circ \varphi_x$. Let $\Cb(x,y)$ be a morphism in $\Cb$, and let $g_x = \varphi_y\circ M(x,y) - \left(\bigoplus_{I\in \mathscr{I}}M_I(x,y)\right)\circ \varphi_x$. Then $\mathrm{id}_\k\otimes_R g_x =\overline{\varphi}_y\circ V(x,y) - \left(\bigoplus_{I\in \mathscr{I}}V_I(x,y)\right)\circ \overline{\varphi}_x =0$. This implies that $\mathrm{im}\, g_x = \m_R\mathrm{im}\,g_x$. By Nakayama's Lemma, we have $\mathrm{im}\, g_x=0$ or $g_x =0$. Thus we obtain a morphism between $R$-generalized persistence modules $\varphi: M\to  \bigoplus_{I\in \mathscr{I}}M_I$. Moroever, since for all $x\in \Ob(\Cb)$, $\overline{\varphi}(x)$ is an isomorphism, it follows again by Nakayama's Lemma that $\varphi(x)$ is also an isomorphism. This implies that $\varphi$ is then an isomorphism between $R$-generalized persistence modules.  Moreover, by Lemma \ref{lemma1.2}, for each $I\in \mathscr{I}$, $\mathbf{End}^\Cb_R(M)=R$. This finishes the proof of Theorem \ref{thm1}. 
\end{proof}

\bibliographystyle{amsplain}
\bibliography{LiftingsPersistence}   

\providecommand{\bysame}{\leavevmode\hbox to3em{\hrulefill}\thinspace}
\providecommand{\MR}{\relax\ifhmode\unskip\space\fi MR }
\providecommand{\MRhref}[2]{%
  \href{http://www.ams.org/mathscinet-getitem?mr=#1}{#2}
}
\providecommand{\href}[2]{#2}
\begin{thebibliography}{1}

\bibitem{bleher8}
F.~M. Bleher and T.~Chinburg, \emph{Universal deformation rings and cyclic
  blocks}, Math. Ann. \textbf{318} (2000), 805--836.

\bibitem{botman1}
M.~B. Botnan and W.~Crawley-Boevey, \emph{Decomposition of persistence
  modules}, Proc. Amer. Math. Soc. \textbf{148} (2020), no.~11, 4581--4596.
  \MR{4143378}

\bibitem{cochoy}
J.~Cochoy and S.~Oudot, \emph{Decomposition of exact pfd persistence
  bimodules}, Discrete Comput. Geom. \textbf{63} (2020), no.~2, 255--293.
  \MR{4057439}

\bibitem{crawley-boevey}
W.~Crawley-Boevey, \emph{Decomposition of pointwise finite-dimensional
  persistence modules}, J. Algebra Appl. \textbf{14} (2015), no.~5, 1550066, 8.
  \MR{3323327}

\bibitem{patel}
A.~Patel, \emph{Generalized persistence diagrams}, J. Appl. Comput. Topol.
  \textbf{1} (2018), no.~3-4, 397--419. \MR{3975559}

\bibitem{sch}
M.~Schlessinger, \emph{Functors of {A}rtin rings}, Trans. Amer. Math. Soc.
  \textbf{130} (1968), 208--222.

\end{thebibliography}

\end{document}